\documentclass[final,leqno]{siamltex}

\usepackage{amsfonts}
\usepackage{arbenson-math}
\usepackage{graphicx}
\usepackage{epstopdf}
\usepackage{amsmath}
\usepackage[ruled,vlined]{algorithm2e}
\usepackage{graphicx}
\usepackage{booktabs}
\usepackage{url}

\usepackage{soulutf8}

\newcommand{\InterListHigh}[1]{I^{#1}}
\newcommand{\ParInterListHigh}[1]{\mathcal{I}^{#1}}
\newcommand{\InterListLow}[1]{I^{#1}} 
\def\listspace{\vspace{0.0in}}
\def\listsep{}
\def\OTree{\mathcal{T}}
\def\PLevel{l_p}
\def\PC{\mathcal{P}}

\def\PWidth{2\sqrt{K}}

\newcommand{\OutDirEqPts}[1]{\BraceOf{y_k^{#1, o, \ell}}}
\newcommand{\OutDirEqChrgs}[1]{\BraceOf{f_k^{#1, o, \ell}}}
\newcommand{\OutDirCheckPts}[1]{\BraceOf{x_k^{#1, o, \ell}}}
\newcommand{\OutDirCheckPots}[1]{\BraceOf{u_k^{#1, o, \ell}}}

\newcommand{\InDirEqPts}[1]{\BraceOf{y_k^{#1, i, \ell}}}
\newcommand{\InDirEqChrgs}[1]{\BraceOf{f_k^{#1, i, \ell}}}
\newcommand{\InDirCheckPts}[1]{\BraceOf{x_k^{#1, i, \ell}}}
\newcommand{\InDirCheckPots}[1]{\BraceOf{u_k^{#1, i, \ell}}}

\newcommand{\OutEqPts}[1]{\BraceOf{y_k^{#1, o}}}
\newcommand{\OutEqChrgs}[1]{\BraceOf{f_k^{#1, o}}}
\newcommand{\OutCheckPts}[1]{\BraceOf{x_k^{#1, o}}}
\newcommand{\OutCheckPots}[1]{\BraceOf{u_k^{#1, o}}}

\newcommand{\InEqPts}[1]{\BraceOf{y_k^{#1, i}}}
\newcommand{\InEqChrgs}[1]{\BraceOf{f_k^{#1, i}}}
\newcommand{\InCheckPts}[1]{\BraceOf{x_k^{#1, i}}}
\newcommand{\InCheckPots}[1]{\BraceOf{u_k^{#1, i}}}

\title{A parallel directional fast multipole method}

\author{
Austin R. Benson\thanks{ICME, Stanford University, 475 Via Ortega, Stanford, CA 94305 (\texttt{arbenson@stanford.edu})}
\and Jack Poulson\thanks{Department of Mathematics, Stanford University, 450 Serra Mall, Stanford, CA 94305 (\texttt{poulson@stanford.edu})}
\and Kenneth Tran\thanks{Microsoft Corporation, One Microsoft Way, Redmond, WA, 98052 (\texttt{ktran@microsoft.com})}
\and \newline Bj\"{o}rn Engquist\thanks{Department of Mathematics and ICES, University of Texas at Austin, 1 University Station
C1200, Austin, TX, 78712 (\texttt{engquist@ices.utexas.edu})}
\and Lexing Ying\thanks{Department of Mathematics and ICME, Stanford University, 450 Serra Mall, Stanford, CA 94305 (\texttt{lexing@math.stanford.edu})}
}
        
\begin{document}

\maketitle

\begin{abstract}
This paper introduces a parallel directional fast multipole method
(FMM) for solving $N$-body problems with highly oscillatory kernels,
with a focus on the Helmholtz kernel in three dimensions.  This class
of oscillatory kernels requires a more restrictive low-rank criterion
than that of the low-frequency regime, and thus effective
parallelizations must adapt to the modified data dependencies.  We
propose a simple partition at a fixed level of the octree and show
that, if the partitions are properly balanced between $p$ processes,
the overall runtime is essentially $\bigoh{N \log N/p + p}$.  By the
structure of the low-rank criterion, we are able to avoid
communication at the top of the octree.  We demonstrate the
effectiveness of our parallelization on several challenging models.
\end{abstract}



\pagestyle{myheadings}
\thispagestyle{plain}

\section{Introduction}

This paper is concerned with a parallel algorithm for the rapid solution of a class of $N$-body problems.
Let $\BraceOf{f_i, 1 \le i \le N}$ be a set of $N$ densities located at points $\BraceOf{p_i, 1 \le i \le N}$ in $\mbb{R}^{3}$, with $\vert p_i \vert \le K / 2$, where $\vert \cdot \vert$ is the Euclidean norm and $K$ is a fixed constant.
Our goal is to compute the potentials $\BraceOf{u_i, 1 \le i \le N}$ defined by
\begin{equation}\label{eq:potentials}
u_i = \ssum{j=1}{N}{G(p_i, p_j) \cdot f_j},
\end{equation}
where $G(x, y) = e^{2\pi \iota \vert x - y \vert} / \vert x - y \vert$
is the Green's function of the Helmholtz equation.  We have scaled the
problem such that the wave length equals one and thus high frequencies
correspond to problems with large computational domains, i.e., large
$K$.

The computation in Equation~\ref{eq:potentials} arises in
electromagnetic and acoustic scattering, where the usual partial
differential equations are transformed into boundary integral
equations (BIEs)~\cite{Colton:2013, Harris:2001, Tsuji:2011,
  Ying:2009}.  The discretized BIEs often result in large, dense
linear systems with $N=O(K^2)$ unknowns, for which iterative methods
are used.  Equation~\ref{eq:potentials} represents the matrix-vector
multiplication at the core of the iterative methods.

\subsection{Previous work}
Direct computation of Equation~\ref{eq:potentials} requires
$\bigoh{N^2}$ operations, which can be prohibitively expensive for
large values of $N$.  A number of methods reduce the computational
complexity without compromising accuracy.  The approaches include
FFT-based methods~\cite{Bleszynski:1996, Bojarski:1972, Bruno:2001},
local Fourier bases~\cite{Bradie:1993, Canning:1992}, and the high
frequency fast multipole method (HF-FMM)~\cite{Cecka:2013, Cheng:2006,
  Rokhlin:1992}.  We focus on a multilevel algorithm developed by
Engquist and Ying which exploits the numerically low-rank interaction
satisfying a \emph{directional parabolic separation
  condition}~\cite{Engquist:2007, Engquist:2009, Engquist:2010,
  Tran:2012}.

There has been much work on parallelizing the FMM for non-oscillatory kernels~\cite{Darve:2011, Greengard:1990, Rahimian:2010, Yokota:2012a}.
In particular, the kernel-independent fast multipole method (KI-FMM)~\cite{Ying:2004} is a central tool for our algorithm in the ``low-frequency regime", which we will discuss in subsection~\ref{subsec:sequential}.
Many efforts have extended KI-FMM to modern, heterogeneous architectures~\cite{Chandramowlishwaran:2012, Chandramowlishwaran:2010, Rahimian:2010, Ying:2003, Yokota:2012b}, and several authors have paralellized~\cite{Ergul:2008, Ergul:2009, Taboada:2010, Waltz:2007} the Multilevel Fast Multipole Algorithm~\cite{Chew:2001, Song:1995, Song:1997}, which is a variant of the HF-FMM.

\subsection{Contributions}
This paper provides the first parallel algorithm for computing Equation~\ref{eq:potentials} based on the directional approach.
Since the top of the octree imposes a well-known bottleneck in parallel FMM algorithms~\cite{Ying:2003},
a central advantage of the directional approach is that the interaction lists of boxes in the octree with width greater than or equal to $\PWidth$ are empty.
This alleviates the bottleneck, as no translations are needed at those levels of the octree.

Our parallel algorithm is based on a partition of the boxes at a fixed level in the octree amongst the processes (see subsection~\ref{subsec:partitioning}).
Due to the simple structure of the sequential algorithm, we can leverage the partition into a flexible parallel algorithm.
The parallel algorithm is described in \S~\ref{sec:parallelization}, and the results of several numerical experiments are in \S~\ref{sec:numerical}.
As will be seen, the algorithm exhibits good strong scaling up to 1024 cores 
for several challenging models.

\section{Preliminaries}

In this section, we review the directional low-rank property of the Helmholtz kernel
and the sequential directional algorithm for computing Equation~\ref{eq:potentials}~\cite{Engquist:2007}.
Proposition~\ref{prop:near_field} in subsection~\ref{subsec:definitions} shows why 
the algorithm can avoid computation at the top levels of the octree.
We will use this to develop the parallel algorithm in \S~\ref{sec:parallelization}.

\subsection{Directional low-rank property}
\label{subsec:dir_lowrank}

\begin{figure}
\centering
\vspace{-0.5cm} 
\includegraphics[scale=0.5]{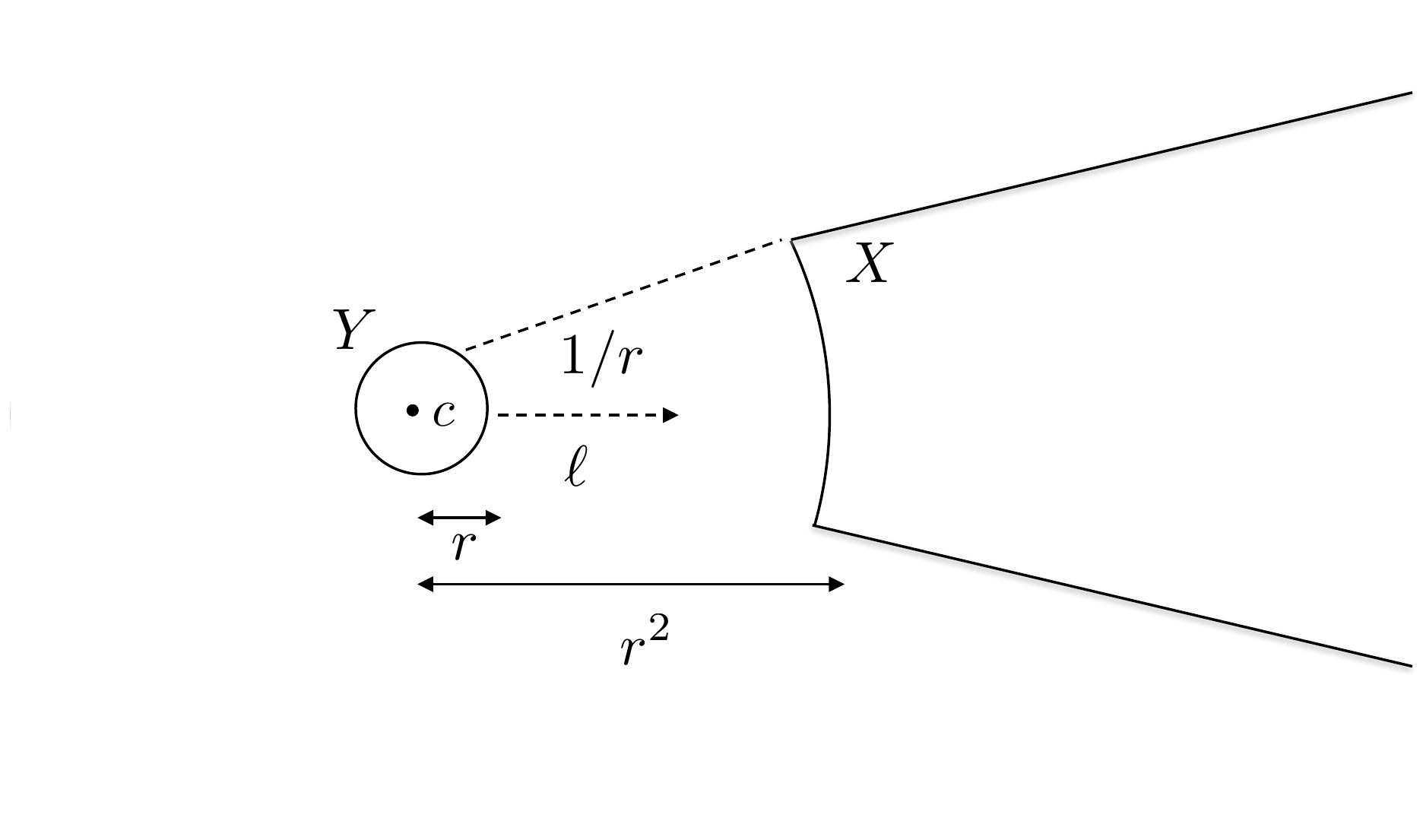}
\vspace{-0.5cm} 
\caption{Sets $Y$ and $X$ that satisfy the directional parabolic separation condition.}
\label{fig:parabolic}
\end{figure}

Let $Y$ be a ball of radius $r \ge \sqrt{3}$ centered at a point $c$ and define the wedge $X = \BraceOf{x : \theta(x, \ell) \le 1/r, \vert x - c \vert \ge r^2}$, where $\ell$ is a unit vector and $\theta(x, \ell)$ is the angle between the vectors $x$ and $\ell$.
$Y$ and $X$ are said to satisfy the \emph{directional parabolic separation condition}, which is illustrated in Fig.~\ref{fig:parabolic}.

Suppose we have a set of charges $\BraceOf{f_i}$ located at points
$\BraceOf{y_i} \subseteq Y$.  Given a threhold $\epsilon$, there exist
two sets of points $\BraceOf{b_q}_{1\le q \le r_\epsilon}$ and
$\BraceOf{a_p}_{1\le p \le r_\epsilon}$ and a set of constants
$\BraceOf{d_{pq}}_{1\le p,q\le r_\epsilon}$ such that, for $x \in X$,
\begin{equation}
\label{eq:equivs1}
\left\vert \ssum{i}{}{G(x, y_i)f_i} -  \ssum{q}{}{G(x, b_q)\left(\ssum{p}{}{d_{qp}\ssum{i}{}{G(a_p, y_i)f_i}}\right)} \right\vert = \bigoh{\epsilon}
\end{equation}
where $r_\epsilon$ is a constant that depends only on $\epsilon$.  A
random sampling method based on pivoted QR
factorization~\cite{Engquist:2010} can be used to determine this
approximation. We can, of course, reverse the roles of the ball and
the wedge.  In Equation~\ref{eq:equivs1}, the quantities
$\BraceOf{\sum_{p}{d_{qp}\sum_{i}{G(a_p, y_i)f_i}}}$, $\BraceOf{b_q}$,
$\BraceOf{\sum_{i}{G(b_q, x_i)f_i}}$, and $\BraceOf{a_p}$ are called
the \emph{directional outgoing} equivalent charges, equivalent points,
check potentials, and check points of $Y$ in direction $\ell$.  When
the roles of $X$ and $Y$ are reversed,
$\BraceOf{\sum_{q}{d_{qp}\sum_{i}{G(b_q, x_i)f_i}}}$, $\BraceOf{a_p}$,
$\BraceOf{\sum_{i}{G(b_q, x_i)f_i}}$, $\BraceOf{b_q}$ are called the
\emph{directional incoming} equivalent charges, equivalent points,
check potentials, and check points of $Y$ in direction $\ell$.  The
constants $d_{pq}$ form a matrix $D=(d_{p,q})$ that we call a
\emph{translation matrix}.

\subsection{Definitions}
\label{subsec:definitions}

Without loss of generality, we assume that the domain width is $K =
2^{2L}$ for a positive integer $L$ (recall that $\vert p_i \vert \le K
/ 2$ in Equation~\ref{eq:potentials}).  The central data structure in
the sequential algorithm is an octree.  The parallel algorithm,
described in \S~\ref{sec:parallelization}, uses the same octree.  For
the remainder of the paper, let $B$ denote a box in the octree and
$w(B)$ the width of $B$.  We say that a box $B$ is in the
high-frequency regime if $w(B) \ge 1$ and in the low-frequency regime
if $w(B) < 1$.  As discussed in~\cite{Engquist:2007}, the octree is
non-adaptive in the non-empty boxes of the high-frequency regime and
adaptive in the low-frequency regime.  In other words, a box $B$ is
partitioned into children boxes if $B$ is non-empty and in the
high-frequency regime or if the number of points in $B$ is greater
than a fixed constant and in the low-frequency regime.

For each box $B$ in the high-frequency regime and each direction $\ell$, we use the following notation for the relevant data:

\listspace\begin{itemize}
\listsep
\item $\OutDirEqPts{B}$, $\OutDirEqChrgs{B}$, $\OutDirCheckPts{B}$, and $\OutDirCheckPots{B}$ are the \emph{outgoing directional} equivalent points, equivalent charges, check points, and check potentials, and

\item  $\InDirEqPts{B}$, $\InDirEqChrgs{B}$, $\InDirCheckPts{B}$, and $\InDirCheckPots{B}$ are the \emph{incoming directional} equivalent points, equivalent charges, check points, and check potentials.
\end{itemize}\listspace

In the low-frequency regime, the algorithm uses the kernel-independent FMM~\cite{Ying:2004}.
For each box $B$ in the low-frequency regime, we use the following notation for the relevant data:

\listspace\begin{itemize}
\listsep
\item $\OutEqPts{B}$, $\OutEqChrgs{B}$, $\OutCheckPts{B}$, and $\OutCheckPots{B}$ are the \emph{outgoing} equivalent points, equivalent charges, check points, and check potentials, and

\item $\InEqPts{B}$, $\InEqChrgs{B}$, $\InCheckPts{B}$, and $\InCheckPots{B}$ are the \emph{incoming} equivalent points, equivalent charges, check points, and check potentials.
\end{itemize}\listspace

Due to the translational and rotational invariance of the relevant kernels, the 
(directional) outgoing and incoming equivalent and check points can be 
efficiently precomputed since they only depend upon the problem size, $K$, and 
the desired accuracy.\footnote{Our current hierarchical wedge decomposition does not respect rotational invariance, and so, in some cases, this precomputation can be asymptotically more expensive than the subsequent directional FMM.}
We discuss the number of equivalent points and check points in subsection~\ref{subsec:precomp}.

Let $B$ be a box with center $c$ and $w(B) \ge 1$.
We define the near field of a box $B$ in the high-frequency regime, $N^B$, as the union of boxes $A$ that contain a point in a ball centered at $c$ with radius $\sqrt{3}w / 2$.
Such boxes are too close to satisfy the directional parabolic condition described in subsection~\ref{subsec:dir_lowrank}.
The far field of a box $B$, $F^B$, is simply the complement of $N^B$.
The interaction list of $B$, $\InterListHigh{B}$, is the set of boxes $N^P \backslash N^B$, where $P$ is the parent box of $B$.

To exploit the low-rank structure of boxes separated by the directional parabolic condition, we need to partition $F^B$ into wedges with a spanning angle of $\bigoh{1 / w}$.
To form the directional wedges, we use the construction described in~\cite{Engquist:2007}.
The construction has a hierarchical property: for any directional index $\ell$ of a box $B$, there exists a directional index $\ell^{\prime}$ for boxes of width $w(B) / 2$ such that the $\ell$-th wedge of $B$ is contained in the $\ell^{\prime}$-th wedge of $B$'s children.
Figure~\ref{fig:hierarchical_wedges} illustrates this idea.
In order to avoid boxes on the boundary of two wedges, the wedges are overlapped.
Boxes that would be on the boundary of two wedges are assigned to the wedge that comes first in an arbitrary lexicographic ordering.

\begin{figure}[tbp]
\centering
\vspace{-2.5cm} 
\includegraphics[scale=0.4]{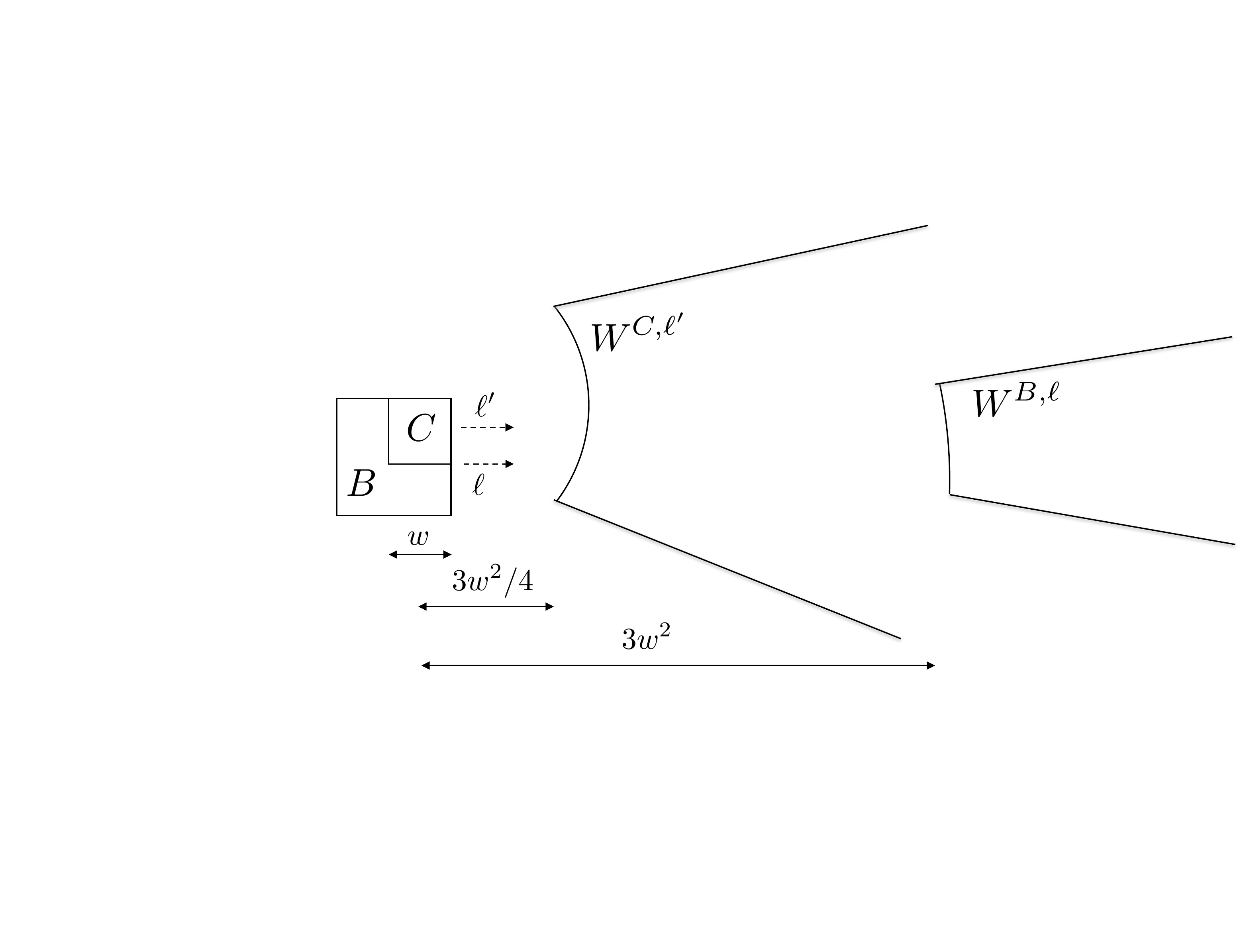}
\vspace{-2.5cm} 
\caption{
A two-dimensional slice of the three-dimensional hierarchical wedges.
$W^{B, \ell}$ is the wedge for box $B$ for the directional index $\ell$.
By the hierarchical construction, we can find an index direction $\ell^{\prime}$ such that
for any child box $C$ of $B$, $W^{B, \ell}$ is contained in $W^{C, \ell^{\prime}}$,
the $\ell^{\prime}$-th wedge of $C$.
Since $C$ has a smaller width ($w$) than $B$ ($2w$), the wedge for $C$ is closer and wider than the wedge for $B$.
This is a consequence of the directional parabolic separation condition described in subsection~\ref{subsec:dir_lowrank}.
}
\label{fig:hierarchical_wedges}
\end{figure}


\subsection{Sequential algorithm}
\label{subsec:sequential}

In the high-frequency regime, the algorithm uses directional M2M, M2L, and L2L translation operators~\cite{Engquist:2007}, similar to the standard FMM~\cite{Cheng:1999}.
The high-frequency directional FMM operators are as follows:

\listspace\begin{itemize}
\listsep
\item HF-M2M constructs the outgoing directional equivalent charges of $B$ from the outgoing equivalent charges of $B$'s children, 
\item HF-L2L constructs the incoming directional check potentials of $B$'s children from the incoming directional check potentials of $B$, and
\item HF-M2L adds to the incoming directional check potentials of $B$ due to outgoing directional charges from each box $A \in \InterListHigh{B}$.
\end{itemize}\listspace

The following proposition shows that at the top levels of the octree, the interaction lists are empty:
\begin{proposition}
\label{prop:near_field}
Let $B$ be a box with width $w$.
If $w \ge 2\sqrt{K}$, then $N^B$ contains all nonempty boxes in the octree.
\end{proposition}
\begin{proof}
Let $c$ be the center of $B$.
Then for any point $p_i \in B$, $|p_i - c| \le \sqrt{3}w/2$.
For any point $p_j$ in a box $A$, $|p_j - c| \le |p_j| + |p_i| + |p_i - c| \le K + \sqrt{3}w / 2$.
$N^B$ contains all boxes $A$ such that there is an $x \in A$ with $|x - c| \le 3w^2/4$.
Thus, $N^B$ will contain all non-empty boxes if
$3w^2/4 \ge \sqrt{3}w / 2 + K$.
This is satisfied for $w \ge \sqrt{4K/3 + 1/3} + \sqrt{1/3}$.
Since $K = 2^{2L}$, this simplifies to $w \ge 2\sqrt{K}$.
\end{proof}

Thus, if $w(B) \ge 2\sqrt{K}$, no HF-M2L translations are performed ($I^B = \emptyset$).
As a consequence, no HF-M2M or HF-M2L translations are needed for those boxes.
This property of the algorithm is crucial for the construction of the parallel algorithm that we present in \S~\ref{sec:parallelization}.

We denote the corresponding low-frequency translation operators as LF-M2M, LF-L2L, and LF-M2L.
We also denote the interaction list of a box in the low frequency regime by $\InterListLow{B}$, although the definition of $N^B$ is different~\cite{Engquist:2007}.
Algorithm~\ref{alg:DirFMMSequential} presents the sequential algorithm in terms of the high- and low-frequency translation operators.

\begin{algorithm}
\caption{Sequential directional FMM for Equation~\ref{eq:potentials}. \label{alg:DirFMMSequential}}
\KwData{points $\BraceOf{p_i}_{1 \le i \le N}$, densities $\BraceOf{f_i}_{1 \le i \le N}$, and octree $\OTree$}
\KwResult{potentials $\BraceOf{u_i}_{1 \le i \le N}$}
\BlankLine
\ForEach{box $B$ in \emph{postorder} traversal of $\OTree$ with $w(B) < 1$}{
  \ForEach{child box $C$ of $B$}{
  Transform $\{f_k^{C, o}\}$ to $\OutEqChrgs{B}$ via LF-M2M \\
  }
}
\ForEach{box $B$ in \emph{postorder} traversal of $\OTree$ with $1 \le w(B) \le \sqrt{K}$}{
    \ForEach{direction $\ell$ of $B$}{
	Let the $\ell^{\prime}$-th wedge of $B$'s children contain the $\ell$-th wedge of $B$ \\ 
	\ForEach{child box $C$ of $B$}{
	  Transform $\{f_k^{C, o, \ell^{\prime}}\}$ to $\OutDirEqChrgs{B}$ via HF-M2M
	  }
    }
}
\ForEach{box $B$ in \emph{preorder} traversal of $\OTree$ with $1 \le w(B) \le \sqrt{K}$}{
   \ForEach{direction $\ell$ of $B$}{
	    \ForEach{box $A \in \InterListHigh{B}$ in direction $\ell$}{
	       Let $\ell^{\prime}$ be the direction for which $B \in \InterListHigh{A}$ \\
	       Transform $\BraceOf{f_k^{A, o, \ell^{\prime}}}$ via HF-M2L and add result to $\InDirCheckPots{B}$
	    }
	    \ForEach{child box $C$ of $B$}{
	      Transform $\InDirCheckPots{B}$ to $\InDirCheckPots{C}$ via HF-L2L
	    }
    }
}
\ForEach{box $B$ in \emph{preorder} traversal of $\OTree$ with $w(B) < 1$}{
  \ForEach{$A \in \InterListLow{B}$}{
    Transform $\OutEqChrgs{A}$ via LF-M2L and add the result to $\InCheckPots{B}$
    }
    Transform $\InCheckPots{B}$ via LF-L2L \\
    \eIf{$B$ is a leaf box}{
       \ForEach{$p_i \in B$}{Add result to $u_i$}
    } {
      \ForEach{child box $C$ of $B$}{Add result to $\InCheckPots{C}$}
    }
}
\end{algorithm}

\section{Parallelization}
\label{sec:parallelization}

The overall structure of our proposed parallel algorithm is similar to the sequential algorithm.
A partition of the boxes at a fixed level in the octree governs the parallelism, and this partition is described in subsection~\ref{subsec:partitioning}.
We analyze the communication and computation complexities in subsections~\ref{subsec:comm_analysis}~and~\ref{subsec:comp_analysis}.
For the remainder of the paper, we use $p$ to denote the total number of processes.

\subsection{Partitioning the octree}
\label{subsec:partitioning}

\begin{figure}[t]
\centering
\includegraphics[scale=0.5]{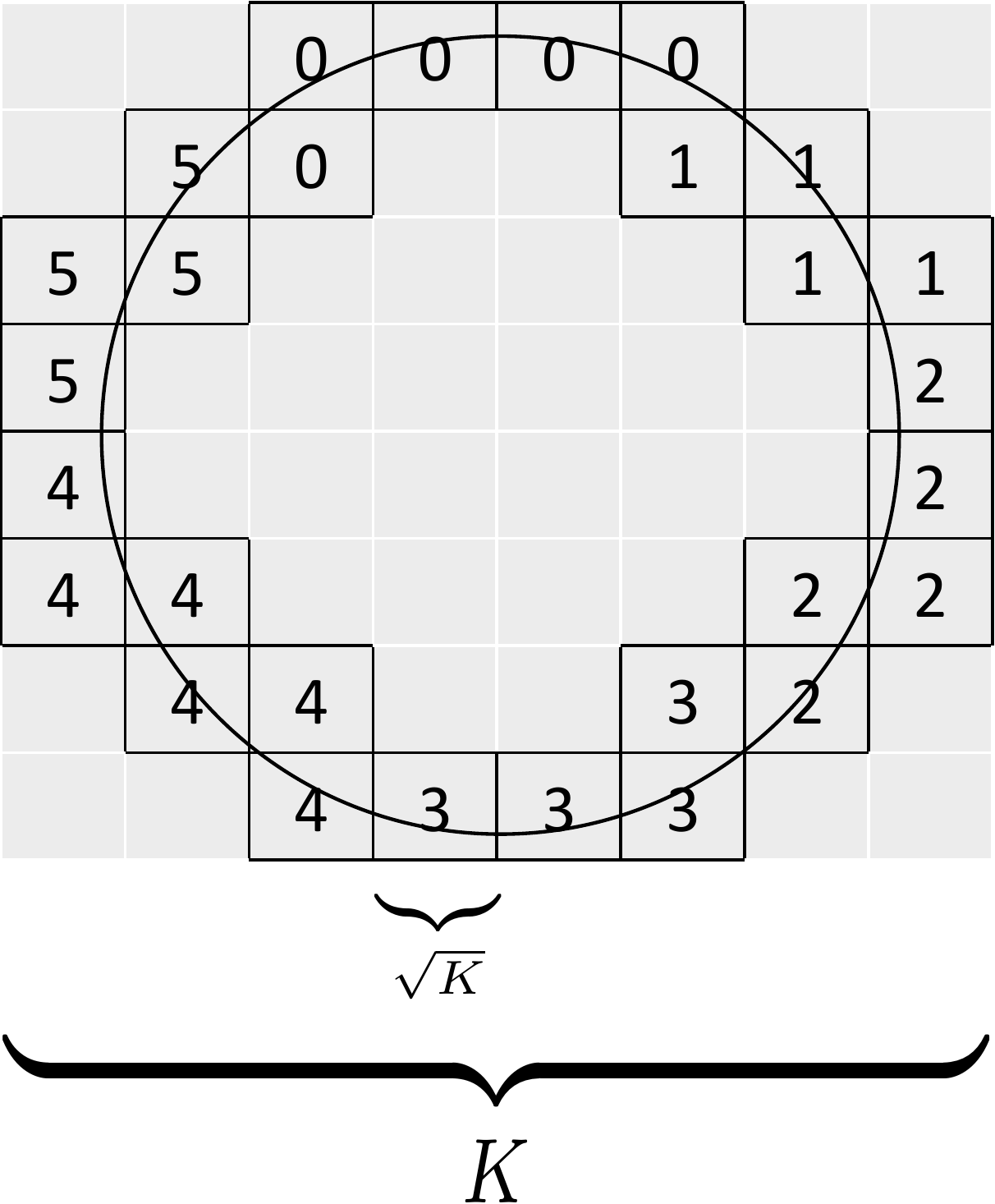}
\caption{
Two-dimensional slice of the octree at the partition level for a spherical geometry with
each box numbered with the process it is assigned to.
The width of the boxes at the partition level is $\sqrt{K}$,
which ensures that the interaction lists of all boxes above this partition level are empty.
We note that the number of boxes with surface points (black border) is much smaller than the total number of boxes.
}
\label{fig:cells}
\end{figure}

To parallelize the computation, we partition the octree at a fixed level.
Number the levels of the octree such that the $l$-th level consists of a $2^{l} \times 2^{l} \times 2^{l}$ grid of boxes of width $K \cdot 2^{-l}$.
By Proposition~\ref{prop:near_field}, there are no dependencies between boxes with width strictly greater than $\sqrt{K}$.
The partition assigns a process to every box $B$ with $w(B) = \sqrt{K}$, and we denote the partition by $\PC$.
In other words, for every non-empty box $B$ at level $L$ in the octree (recall that $K = 2^{2L}$), $\PC(B) \in \BraceOf{0, 1, \ldots, p-1}$.
We refer to level $L$ as the \emph{partition level}.
Process $\PC(B)$ is responsible for all computations associated with box $B$ and the descendent boxes of $B$ in the octree, and so we also define $\mathcal{P}(A) \equiv \mathcal{P}(B)$ for any box $A \subset B$.
An example partition is illustrated in Figure~\ref{fig:cells}.

\subsection{Communication}

\begin{figure}
\centering
\includegraphics[scale=0.6]{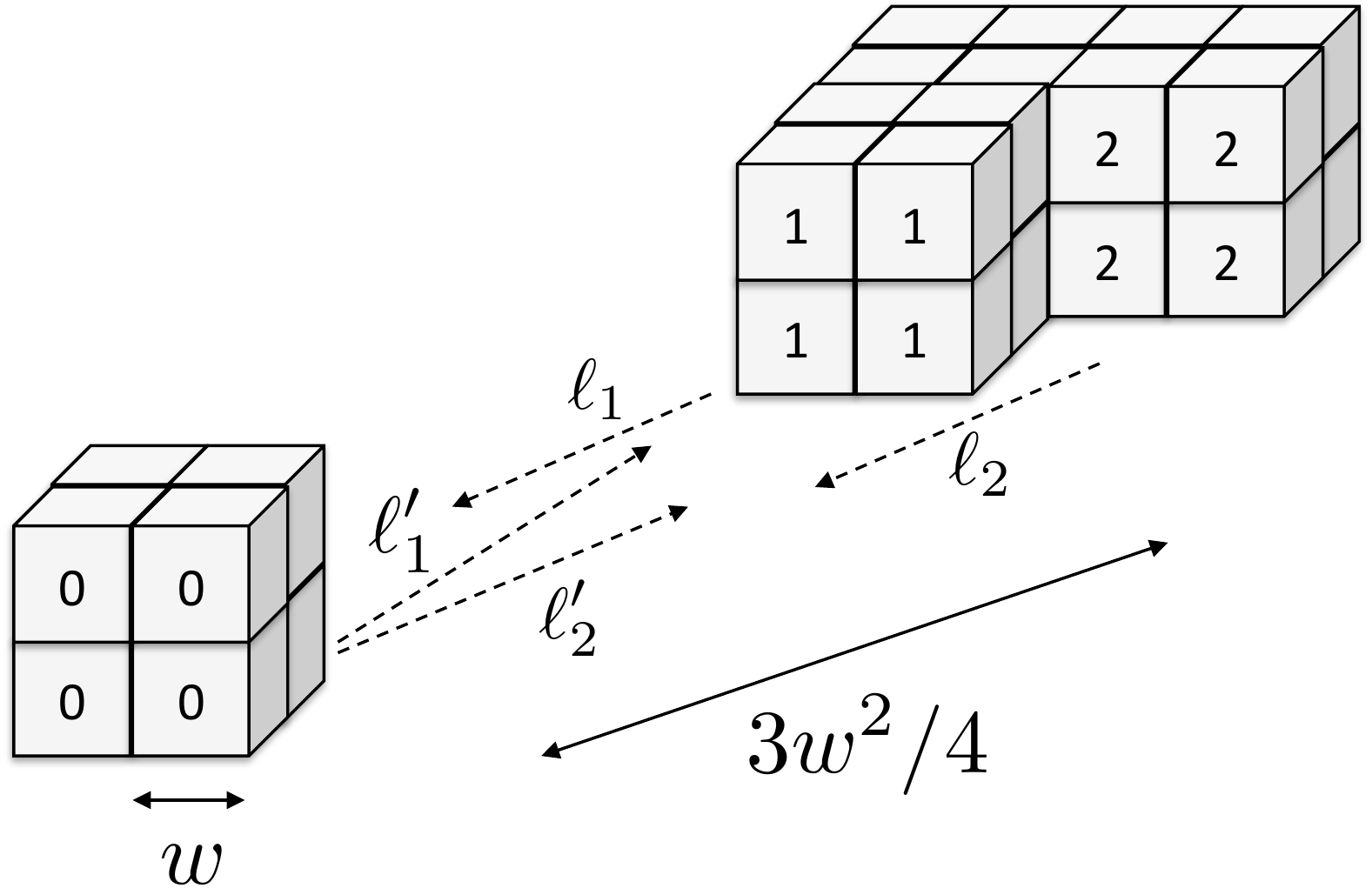}
\caption{
Interaction in the high-frequency regime that induces communication in the parallel algorithm.
In this case, process $0$ needs outgoing directional equivalent densities in direction $\ell_1$ from process $1$ and direction $\ell_2$ from process $2$.
The equivalent densities are translated to incoming directional potentials in the directions $\ell_1^{\prime}$ and $\ell_2^{\prime}$ via HF-M2L.
Similarly, processes $1$ and $2$ needs outgoing directional equivalent densities in direction $\ell_1^{\prime}$ and $\ell_2^{\prime}$ from process $0$.
}
\label{fig:interactions}
\end{figure}

We assume that the ownership of point and densities conforms to the partition $\PC$,
i.e., each process owns exactly the points and densities needed for its computations.
In practice, arbitrary point and density distributions can be handled with a simple pre-processing stage.
Our later experiments make use of k-means clustering~\cite{Macqueen:1967} to distribute the points to processes.

The partition $\PC$ dictates the communication in our algorithm.
Since ${\PC(C) = \PC(B)}$ for any any child box $C$ of $B$, there is no communication for the L2L or M2M translations in both the-low frequency and high-frequency regimes.
These translations are performed in serial on each process.
All communication is induced by the M2L translations.
In particular, for every box $A \in I^B$ for which $\PC(A) \neq \PC(B)$, process $\PC(B)$ needs outgoing (directional) equivalent densities of box $A$ from process $\PC(A)$.
We define $\ParInterListHigh{B} = \BraceOf{A \in I^B \vert \PC(A) \neq \PC(B)}$, i.e., the set of boxes in $B$'s interaction list that do not belong to the same process as $B$.
The parallel algorithm performs two bulk communication steps:
\begin{enumerate}
\item After the HF-M2M translations are complete, all processes exchange outgoing directional equivalent densities.
The interactions that induce this communication are illustrated in Figure~\ref{fig:interactions}.
For every box $B$ in the high-frequency regime and for each box $A \in \ParInterListHigh{B}$, process $\PC(B)$ requests $f^{A, o, \ell}$ from process $\PC(A)$, where $\ell$ is the index of the wedge for which $B \in I^A$.
\item After the HF-L2L translations are complete, all processes exchange outgoing non-directional equivalent densities.
For every box $B$ in the low-frequency regime and for each box $A \in \ParInterListHigh{B}$, process $\PC(B)$ requests $f^{A, o}$ from process $\PC(A)$.
\end{enumerate}

\subsection{Parallel algorithm}

We finally arrive at our proposed parallel algorithm.
Communication for the HF-M2L translations occurs after the HF-M2M translations are complete, and communication for the LF-M2L translations occurs after the HF-L2L translations are complete.
The parallel algorithm from the view of process $q$ is in Algorithm~\ref{alg:DirFMMParallel}.
We assume that communication is not overlapped with computation.
Although there is opportunity to hide communication by overlapping, the overall communication time is small in practice, and the more pressing concern is proper load balancing (see \S~\ref{sec:numerical}).

\begin{algorithm}[tbp]
\caption{Parallel directional FMM for Equation~\ref{eq:potentials} from the view of process $q$.\label{alg:DirFMMParallel}}
\KwData{points $\BraceOf{p_i}_{1 \le i \le N}$, densities $\BraceOf{f_i}_{1 \le i \le N}$, octree $\OTree$, partition $\PC$}
\KwResult{potentials $\BraceOf{u_i}_{1 \le i \le N}$}
\BlankLine
\ForEach{box $B$ in \emph{postorder} traversal of $\OTree$ with $w(B) < 1, \PC(B) = q$}{
  \ForEach{child box $C$ of $B$}{
  \texttt{// $\mathcal{P}(C) = \mathcal{P}(B) = q$, so child data is available locally}
  Transform $\{f_k^{C, o}\}$ to $\OutEqChrgs{B}$ via LF-M2M \\
  }
}
\ForEach{box $B$ in \emph{postorder} traversal of $\OTree$ with $1 \le w(B) \le \sqrt{K}, \PC(B) = q$}{
   \ForEach{direction $\ell$ of $B$}{
	Let the $\ell^{\prime}$-th wedge of $B$'s children contain the $\ell$-th wedge of $B$ \\ 
	\ForEach{child box $C$ of $B$}{
           \texttt{// $\mathcal{P}(C) = \mathcal{P}(B) = q$, so child data is available locally}
            Transform $\{f_k^{C, o, \ell^{\prime}}\}$ to $\OutDirEqChrgs{B}$ via HF-M2M
         }   
    }
}
\texttt{// High-frequency M2L data communication} \\
\ForEach{box $B$ with $w(B) > 1$ and $\PC(B) = q$} {
  \ForEach{box $A \in \ParInterListHigh{A}$} {
	  Let $\ell$ be the direction for which $B \in I^A$ \\
	  Request $f^{A, o, \ell}$ from process $\PC(A)$
  }
}
\ForEach{box $B$ in \emph{preorder} traversal of $\OTree$ with $1 \le w(B) \le \sqrt{K}, \PC(B) = q$}{
   \ForEach{direction $\ell$ of $B$}{
        \ForEach{box $A \in \InterListHigh{B}$ in direction $\ell$}{
	       Let $\ell^{\prime}$ be the direction for which $B \in \InterListHigh{A}$ \\
	       Transform $\BraceOf{f_k^{A, o, \ell^{\prime}}}$ via HF-M2L and add result to $\InDirCheckPots{B}$
        }
        \ForEach{child box $C$ of $B$}{
          \texttt{// $\mathcal{P}(C) = \mathcal{P}(B) = q$, so child data is available locally}
          Transform $\InDirCheckPots{B}$ to $\InDirCheckPots{C}$ via HF-L2L.
        }
    }
}
\texttt{// Low-frequency M2L data communication} \\
\ForEach{box $B$ with $w(B) > 1$ and $\PC(B) = q$} {
  \ForEach{box $A \in \ParInterListHigh{A}$} {
    Request $f^{A, o}$ from process $\PC(A)$
  }
}
\ForEach{box $B$ in \emph{preorder} traversal of $\OTree$ with $w(B) < 1, \PC(B) = q$}{
  \ForEach{$A \in \InterListLow{B}$}{
    Transform $\OutEqChrgs{A}$ via LF-M2L and add the result to $\InCheckPots{B}$
    }
    Transform $\InCheckPots{B}$ via LF-L2L \\
    \eIf{$B$ is a leaf box}{
       \ForEach{$p_i \in B$}{Add result to $u_i$}
    } {
      \texttt{// $\mathcal{P}(C) = \mathcal{P}(B) = q$, so child data is available locally}
      \ForEach{child box $C$ of $B$}{Add result to $\InCheckPots{C}$}
    }
}
\end{algorithm}

\subsection{Communication complexity}
\label{subsec:comm_analysis}

We will now discuss the communication complexity of our algorithm
assuming that the $N$ points are quasi-uniformly sampled from a
two-dimensional manifold with $N = \bigoh{K^2}$, which is generally
satisfied for many applications which involve boundary integral
formulations.

\begin{proposition}
\label{prop:comm}
  Let $\mathbb{S}$ be a surface in $B(0,1/2)$. Suppose that for a fixed $K$,
  the points $\{p_i, 1 \le i \le N\}$ are samples of $K\mathbb{S}$, where
  $N=\bigoh{K^2}$ and $K\mathbb{S} = \{K \cdot p \vert p\in \mathbb{S}\}$ 
  (the surface obtained by magnifying $\mathbb{S}$ by a factor of $K$). 
  Then, for any prescribed accuracy, the proposed algorithm sends at most
  $\bigoh{K^2 \log K} = \bigoh{N\log N}$ data in aggregate.
\end{proposition}

\begin{proof}[Outline of the proof]
  We analyze the amount of data communicated at each of two steps.
\begin{itemize}
         
\item 

We claim that at most $\bigoh{N \log N}$ data is communicated for HF-M2L.
We know that, for a box $B$ of width $w$, the boxes in its high-frequency interaction list are approximately located in a ball centered at $B$ with radius $3w^2/4$ and that there are $\bigoh{w^2}$ directions $\BraceOf{\ell}$.
The fact that our points are samples from a two-dimensional manifold implies that there are at most $\bigoh{w^4/w^2} = \bigoh{w^2}$ boxes in $B$'s interaction list.
Each box in the interaction list contributes at most a constant amount of communication.
Since the points are sampled from a two-dimensional manifold, there are $\bigoh{K^2 / w^2}$ boxes of width $w$.
Thus, each level in the high-frequency regime contributes at most $\bigoh{K^2}$ communication.
We have $\bigoh{\log K}$ levels in the high-frequency regime ($w = 1, 2, 4, \ldots, \sqrt{K}$), so the total communication volume for HF-M2L is at most $\bigoh{K^2\log K} = \bigoh{N \log N}$.

\item

In the low-frequency regime, there are $\bigoh{N}$ boxes, and the interaction list of each box is of a constant size.
Thus, at most $\bigoh{N}$ data is communicated for LF-M2L.

\end{itemize} \end{proof}

To analyze the cost of communication, we use a common model: each process can simultaneously send and receive a single message at a time, and the cost of sending a message with $n$ units of data is $\alpha + \beta n$~\cite{Ballard:2011, Chan:2007, Thakur:2005}.
$\alpha$ is the start-up cost, or latency, and $\beta$ is the inverse bandwidth.
At each of the two communication steps, each process needs to communicate with every other process.
If  the communication is well-balanced at each level of the octree, i.e., the amount of data communicated between any two processes is roughly the same, then the communication is similar to an ``all-to-all" or ``index" pattern.
The communication time is then $\bigoh{p\alpha + \frac{1}{p}K^2\log K \beta}$ or $\bigoh{\log p \alpha + \frac{\log p}{p} K^2\log K \beta}$ using standard algorithms~\cite{Bruck:1997, Thakur:2005}.
We emphasize that, while we have not shown that such balanced partitions even 
always exist, if the partition $\PC$ clusters the assignment of boxes to processes, then many boxes in the interaction lists are already owned by the correct process.
This tends to reduce communication, and in practice, the actual communication time is small (see subsection~\ref{subsec:strong}).

\subsection{Computational complexity}
\label{subsec:comp_analysis}

Our algorithm parallelizes all of the translation operators, which constitute nearly all of the computations.
This is a direct result from the assignment of boxes to processes by $\PC$.
However, the computational complexity analysis depends heavily on the partition.
We will make idealizing assumptions on the partitions to show how the complexity changes with the number of processes, $p$.

Recall that we assume that communication and computation are not overlapped.
Thus, the translations are computed simultaneously on all processes.

\begin{proposition}
\label{prop:LF_comp}
Suppose that each process has the same number boxes in the low-frequency regime.
Then the LF-M2M, LF-M2L, and LF-L2L translations have computational complexity $\bigoh{N / p}$.
\end{proposition}
\begin{proof}
Each non-empty box requires a constant amount of work for each of the low-frequency translation operators since the interaction lists are of constant size.
There are $\bigoh{N}$ boxes in the low-frequency regime, so the computational complexity is $\bigoh{N / p}$ if each process has the same number of boxes.
\end{proof}

We note that Proposition~\ref{prop:LF_comp} makes no assumption on the assignment of the computation to process at any fixed level in the octree:
we only need a load-balanced partition throughout \emph{all} of the boxes in the low-frequency regime.
It is important to note that empty boxes (i.e., boxes that do not contain any of the $\BraceOf{p_i}$) require no computation.
When dealing with two-dimensional surfaces in three dimensions, this poses a scalability issue (see subsection~\ref{subsec:scalability}).

For the high-frequency translation operators,
instead of a load balancing the total number of boxes,
we need to load balance the number of directions.
We emphasize that $\PC$ partitions boxes.
Thus, if $\PC(B) = q$, process $q$ is assigned all directions for the box $B$.

\begin{proposition}
\label{prop:HF_comp}
Suppose that each process is assigned the same number of directions and that the assumptions of Proposition~\ref{prop:comm} are satisfied.
Then the HF-M2M, HF-M2L, and HF-L2L translations have computational complexity $\bigoh{N\log N / p}$.
\end{proposition}

\begin{proof}
For HF-M2M, each box of width $w$ contains $\bigoh{w^2}$ directions $\BraceOf{\ell}$.
Each direction takes $\bigoh{1}$ computations.
There are $\bigoh{\log K}$ levels in the high-frequency regime.
Thus, in parallel, HF-M2M translations take $\bigoh{N \log N / p}$ operations provided that each process has the same number of directions.
The analysis for HF-L2L is analogous.

Now consider HF-M2L.
Following the proof of Proposition~\ref{prop:comm}, there are $\bigoh{w^2}$ directions $\BraceOf{\ell}$ for a box $B$ of width $w$.
For each direction, there is $\bigoh{1}$ boxes in $B$'s interaction list, and each interaction is $\bigoh{1}$ work.
Thus, each direction $\ell$ induces $\bigoh{1}$ computations.
There are $\bigoh{K^2 \log K} = \bigoh{N \log N}$ directions.
Therefore, in parallel, HF-M2L translations take $\bigoh{N \log N / p}$ operations provided that each process has the same number of directions.
\end{proof}

The asymptotic analysis absorbs several constants that are important
for understanding the algorithm.  In particular, the size of the
translation matrix $D$ in Equation~\ref{eq:equivs1} drives the running
time of the costly HF-M2L translations.  The size of $D$ depends on
the box width $w$, the direction $\ell$, and the target accuracy
$\epsilon$.  In subsection~\ref{subsec:precomp}, we discuss the size
of $D$ for problems used in our numerical experiments.

\subsection{Scalability}
\label{subsec:scalability}

Recall that a single process performs the computations associated with all descendent boxes of a given box at the partition level.
Thus, the scalability of the parallel algorithm is limited by the number of non-empty boxes in the octree at the partition level.
If there are $n$ boxes at the partition level that contain points, then the maximum speedup is $n$.
Since the boxes at the partition level can have width as small as $\bigoh{\sqrt{K}}$, there are $\bigoh{(K / \sqrt{K})^3} = \bigoh{K^{3/2}}$ total boxes at the partition level.
However, in the scattering problems of interest, the points are sampled from a two-dimensional surface and are highly non-uniform.
In this setting, the number of nonempty boxes of width $w$ is $\bigoh{K^2 / w^2}$.
The boxes at the partition level have width $\bigoh{\sqrt{K}}$, so the number of nonempty boxes is $\bigoh{K}$.
If $p = \bigoh{K}$, then the communication complexity is $\bigoh{\alpha K + K\log K\beta} = \bigoh{\alpha \sqrt{N} + \sqrt{N}\log N \beta}$, and the computational complexity is $\bigoh{\sqrt{N}\log N}$.

The scalability of our algorithm is limited by the geometry of our objects, which we observe in subsection~\ref{subsec:strong}.
However, we are still able to achieve significant strong scaling with this limitation.

\section{Numerical results}
\label{sec:numerical}

We now illustrate how the algorithm scales through a number of numerical experiments.
Our algorithm is implemented with C++, and the communication is implemented with the Message Passing Interface (MPI).
All experiments were conducted on the Lonestar supercomputer at the Texas Advanced Computing Center.
Each node has two Xeon Intel Hexa-Core 3.3 GHz processes (i.e., 12 total cores) and 24 GB of DRAM, and the nodes are interconnected with InfiniBand.
Each experiment used four MPI processes per node.
The implementation of our algorithm is available at \url{https://github.com/arbenson/ddfmm}.

\subsection{Pre-computation and separation rank}
\label{subsec:precomp}

\begin{table}[tbp]
\centering
\caption{
Maximum dimension of the translation matrix $D$ from Equation~\ref{eq:equivs1} for each target accuracy $\epsilon$ and box width $w$.
Given $\epsilon$, the maximum dimension is approximately the same for each box width.
}
\begin{tabular}{l c c c c c}
\toprule
 &  $w = 1$ & $w = 2$ & $w = 4$ & $w = 8$ & $w = 16$ \\
\midrule
$\epsilon = $1e-4 & 60 & 58 & 60 & 63 & 67 \\
$\epsilon = $1e-6 & 108 & 93 & 96 & 95 & 103 \\
$\epsilon = $1e-8 & 176 & 142 & 141 & 136 & 144\\
\bottomrule
\end{tabular}
\label{tab:seprank}
\end{table}

The pre-computation for the algorithm is the computation of the
matrices $D=(d_{pq})_{1\le p,q\le r_\epsilon}$ used in
Equation~\ref{eq:equivs1}.  For each box width $w$, there are
approximately $2w^2$ directions $\ell$ that cover the directional
interactions.  Each direction at each box width requires the
high-frequency translation matrix $D$.  The dimension of $D$ is called
the \emph{separation rank}.  The maximum separation rank for each box
width $w$ and target accuracy $\epsilon$ is in
Table~\ref{tab:seprank}.  We see that the separation rank is
approximately constant when $\epsilon$ is fixed.  As $\epsilon$
increases, our algorithm is more accurate and needs larger translation
matrices.

Table~\ref{tab:precomp_mem} lists the size of the data files used in our experiments that store the translation matrices for the high- and low-frequency regimes.
Each process stores this data in memory.

\begin{table}[tbp]
\centering
\caption{
Size of files that store pre-computed  data for the parallel algorithm.
The data are suitable for $K < 1024$.  For larger values of $K$, translation matrices for more boxes would be needed.
}
\begin{tabular}{l c c c c c c c}
\toprule
& $\epsilon=$ 1e-4 & $\epsilon=$1e-6 & $\epsilon=$1e-8 \\
\midrule
High frequency & 89 MB & 218 MB & 440 MB \\
Low frequency & 19 MB & 64 MB & 158 MB \\
\bottomrule
\end{tabular}
\label{tab:precomp_mem}
\end{table}

\subsection{Strong scaling on different geometries}
\label{subsec:strong}

\begin{figure}[tbp]
	\centering
	\includegraphics[scale=0.08]{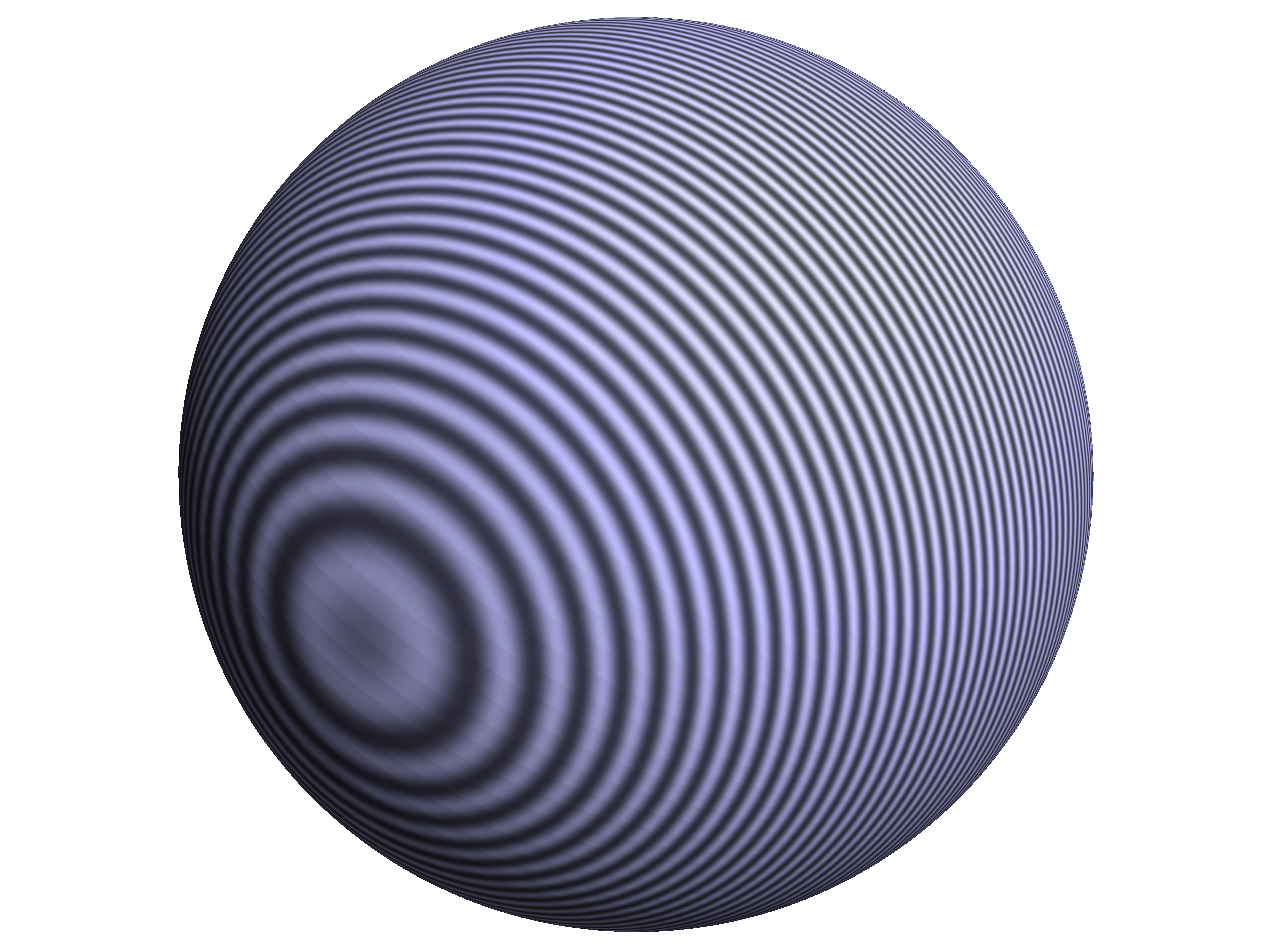}
	\includegraphics[scale=0.08]{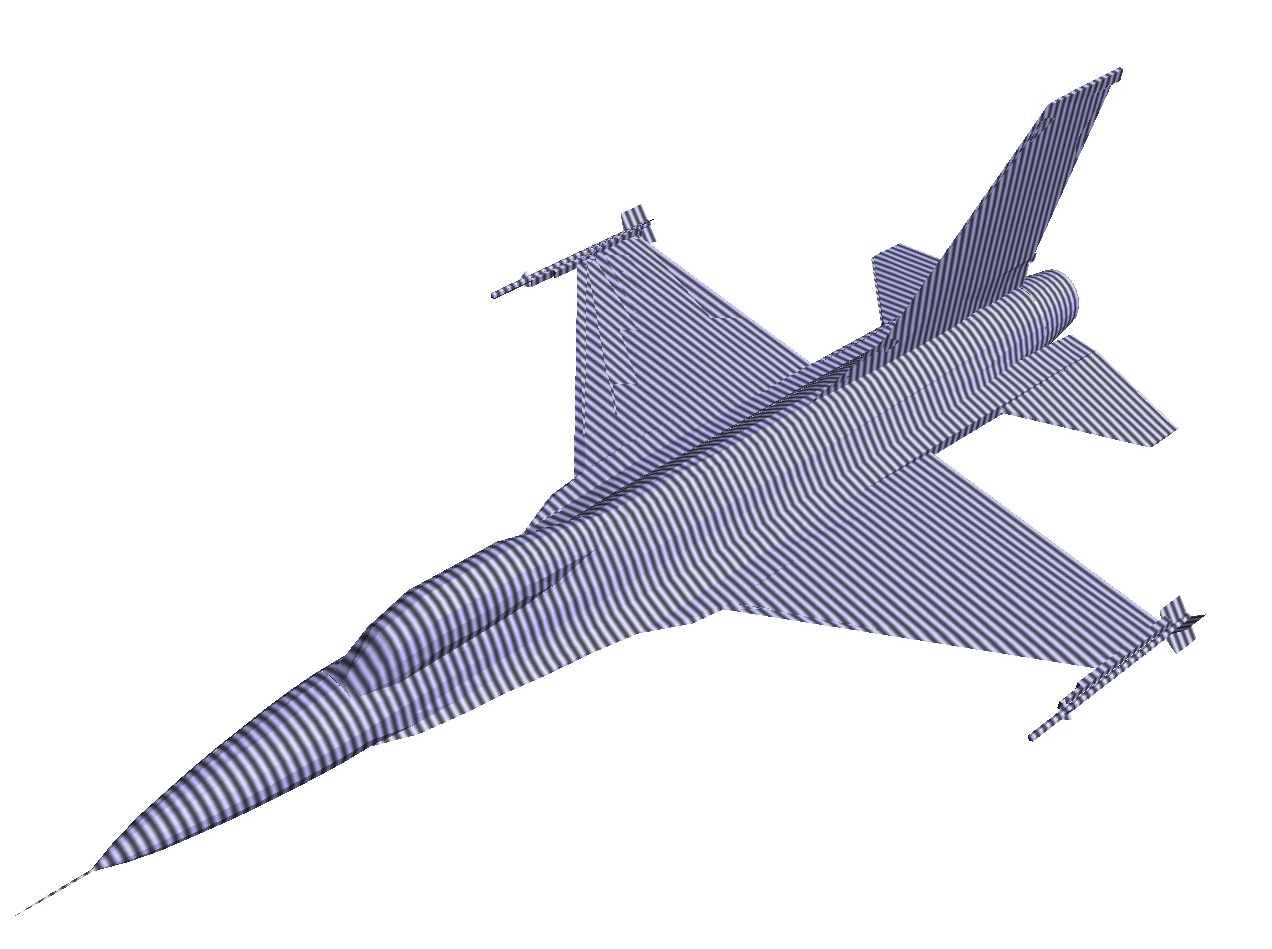}
	\includegraphics[scale=0.08]{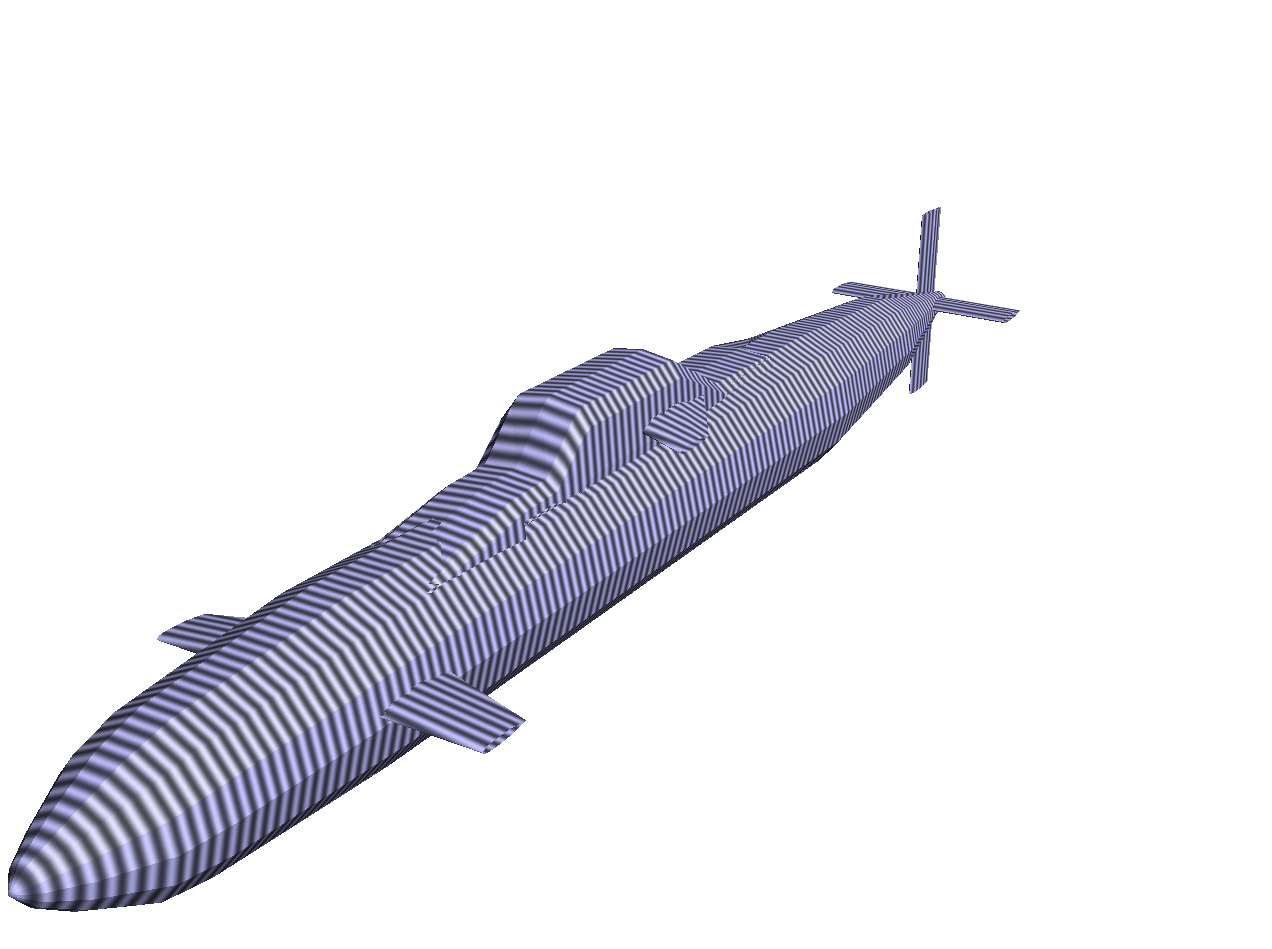}
	\caption{Sphere (left), F16 (middle), and Submarine (right) geometries.}
	\label{fig:geoms}
\end{figure}

We test our algorithm on three geometries: a sphere, an F16 fighter aircraft, and a submarine.
The geometries are in Figure~\ref{fig:geoms}.
The surface of each geometry is represented by a triangular mesh, and the point set $\BraceOf{p_i}$ is generated by sampling the triangular mesh randomly with $10$ points per wavelength on average.
The densities $\BraceOf{f_i}$ are generated by independently sampling from a standard normal distribution.

In our experiments, we use simple clustering techniques to partition the data and create the partition $\PC$.
To partition the points, we use k-means clustering, ignoring the octree boxes where the points live.
To create $\PC$, we use an algorithm that greedily assigns a box to the process that owns the most points in the box while also balancing the number of boxes on each process.
Algorithm~\ref{alg:comp_partition} formally describes the construction of $\PC$.
Because the points are partitioned with k-means clustering, Algorithm~\ref{alg:comp_partition} outputs a map $\PC$ that tends to cluster the assignment of process to box as well.

\begin{algorithm}[tbph]
\caption{Construction of $\PC$ used in numerical experiments. \label{alg:comp_partition}}
\KwResult{partition $\PC$}
\BlankLine
Compute $B_{max} = \lceil \# \left(\text{non-empty boxes at partition level}\right) / p \rceil$ \\
$C = \BraceOf{0, 1, \ldots, p-1}$ \\
\For{$i = 0$ to $p - 1$}{
$B_i = 0$
}
\ForEach{non-empty box $B$ at level $\PLevel$}{
  $i = \arg\max_{j \in C} \left|\BraceOf{p_k \in B \vert p_k \text{ on process } j}\right|$ \\
  $\PC(B) = i$ \\
  $B_i = B_i + 1$ \\
  \If{$B_i == B_{max}$}{
    $C = C \backslash \BraceOf{i}$
  }
}
\end{algorithm}

\begin{figure}[tbp]
\centering
\includegraphics[width=0.6\textwidth]{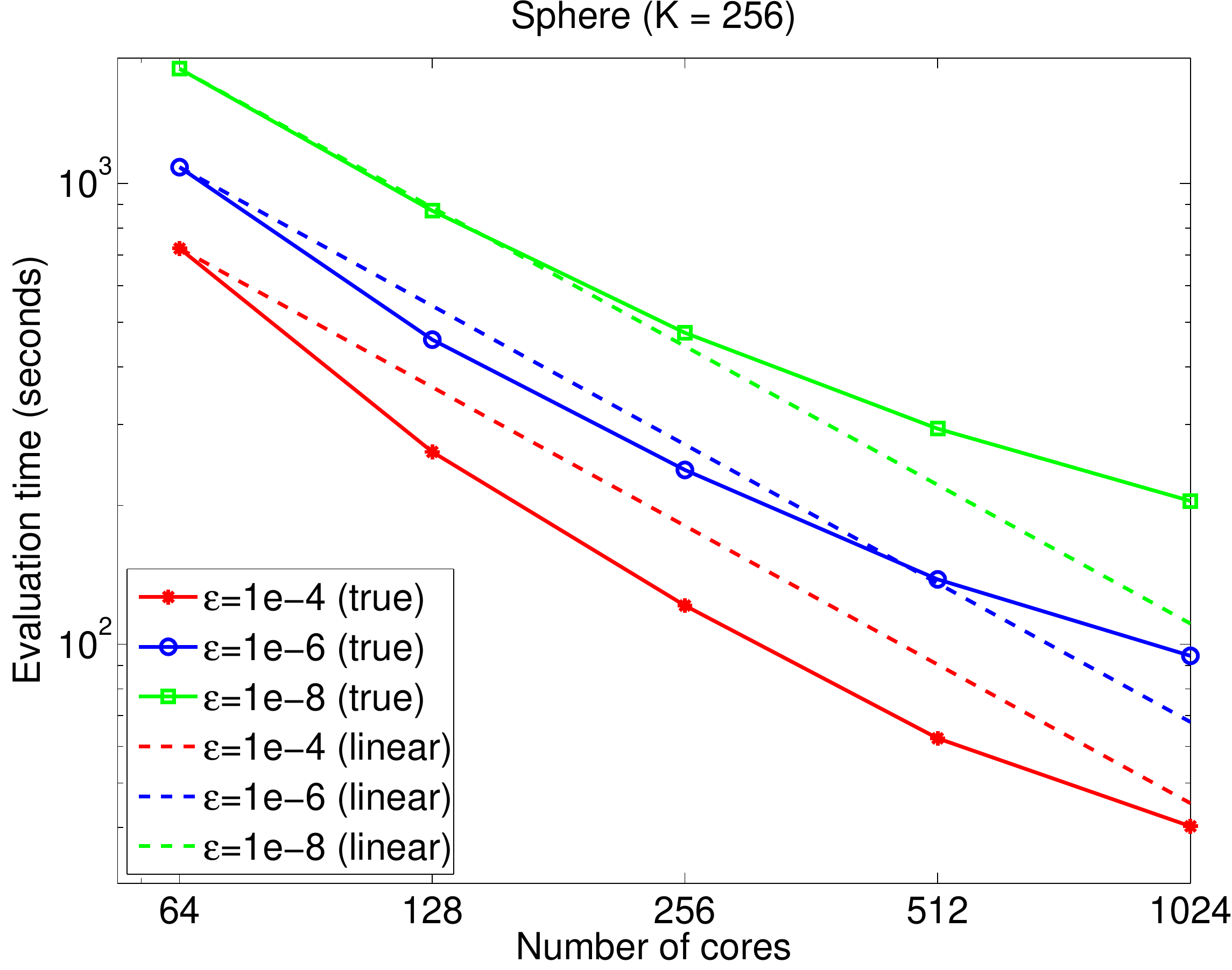} \\
\includegraphics[width=0.6\textwidth]{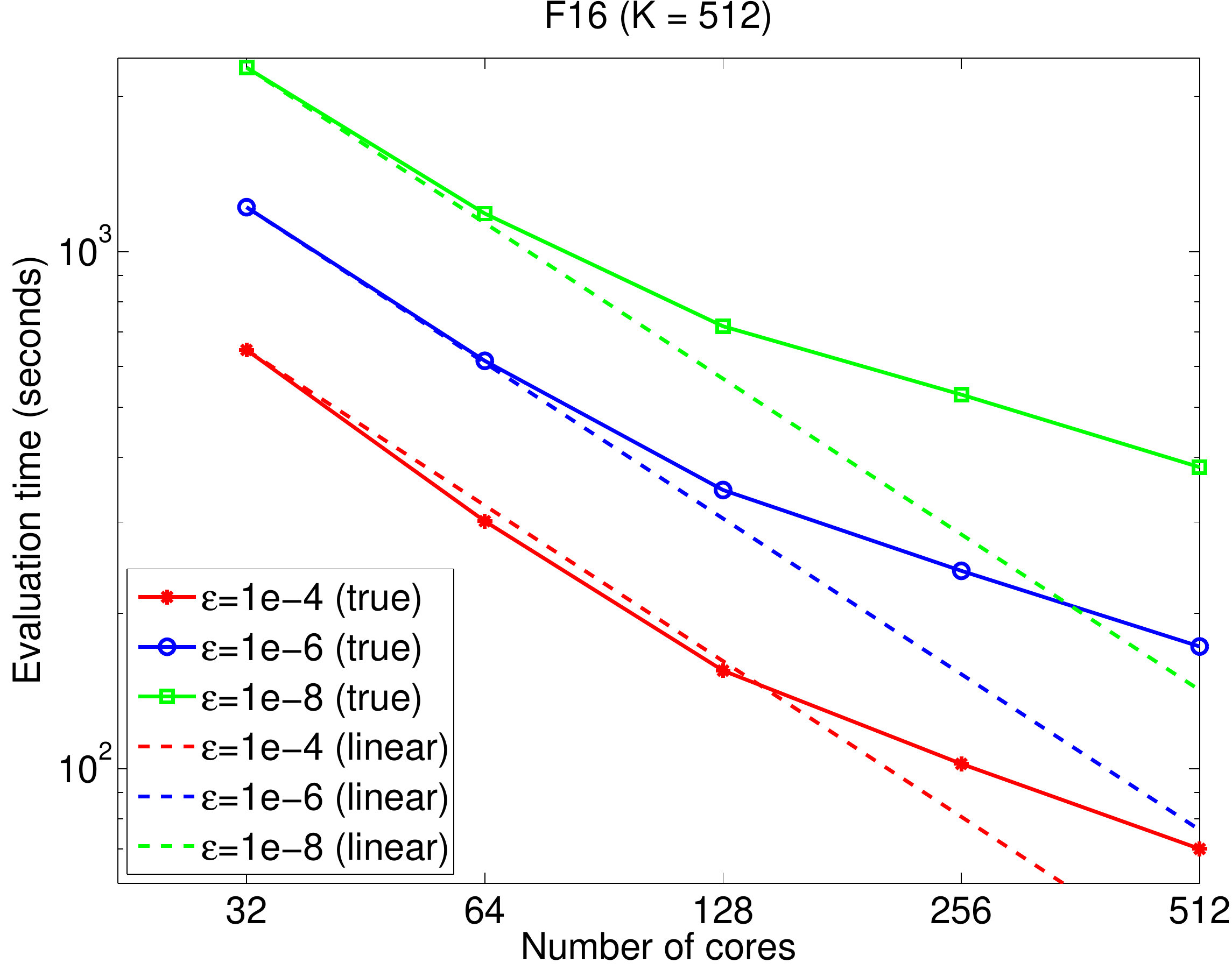}
\includegraphics[width=0.6\textwidth]{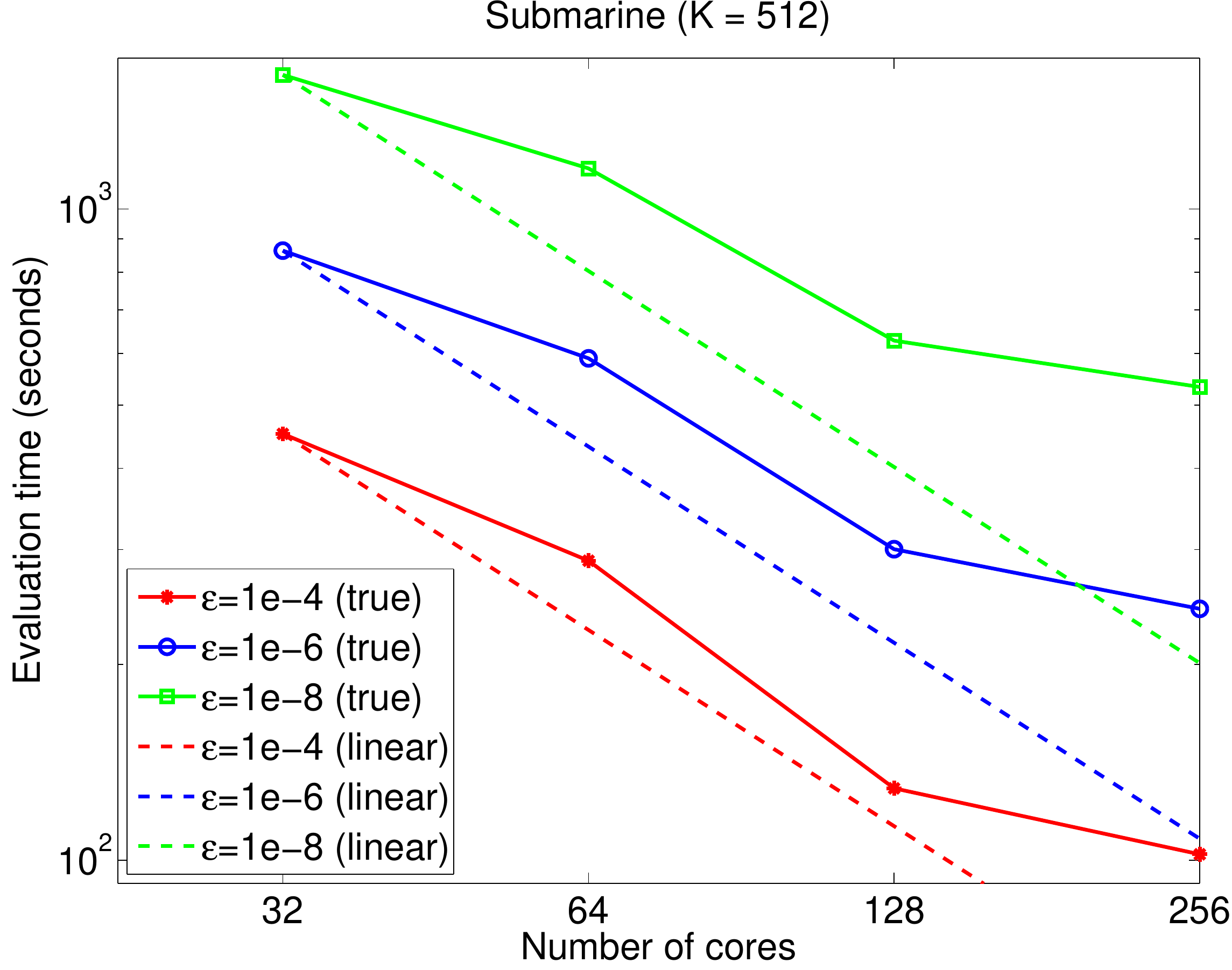}
\caption{
Weak scaling experiments for Sphere ($K = 256$), F16 ($K = 512$), and Submarine ($K = 512$).  Dashed lines are linear scaling from the running time with the smallest number of cores.
}
\label{fig:strong}
\end{figure}

\begin{table}[tbp]
\centering
\caption{
Number of non-empty boxes at the partition level $\PLevel$.
As discussed in subsection~\ref{subsec:scalability}, this limits scalability.
In all numerical experiments for all geometries, $\PLevel = 5$.
}
\begin{tabular}{l c c }
\toprule
Geometry & \#(non-empty boxes) & \#(boxes) \\
\midrule
Sphere & 2,136 & 32,768\\
F16 & 342 & 32,768\\
Submarine & 176 & 32768\\
\bottomrule
\end{tabular}
\label{tab:nonempty}
\end{table}

Figure~\ref{fig:strong} shows strong scaling results for the algorithm
on the geometries with target accuracy $\epsilon = $1e-4, 1e-6, and
1e-8.  The smallest number of cores used for each geometry is chosen
as the smallest power of two such that the algorithm does not run out
of memory when $\epsilon = $ 1e-8.  Our algorithm scales best for the
sphere geometry.  This makes sense because the symmetry of the
geometry balances the number of directions assigned to each process.
While we use a power-of-two number of processes in the scaling
experiments, we note that our implementation works for an arbitrary
number of processes.

Scaling levels off at 512 cores for the F16 and 256 cores for the
Submarine.  This is a direct consequence of the number of non-empty
boxes in the octree at the partition level, which was discussed in
subsection~\ref{subsec:scalability}.  Table~\ref{tab:nonempty} lists
the number of nonempty boxes for the three geometries.

\begin{table}[tbp]
\centering
\caption{
Breakdown of running times for the parallel algorithm.
All listed times are in seconds.
For all geometries, the more accurate solutions produce worse scalings}
\begin{tabular}{l c c c c c c c c}
\toprule
Geometry (K) & $\epsilon$ & p & HF-M2M & HF-M2L& LF-M2L& HF+LF & Total & Efficiency \\
& & & & + HF-L2L & + LF-L2L & Comm. &  \\
\midrule
Sphere (256) & 1e-4 & 64 & 29.44 & 525.69 & 159.83 & 5.53 & 722.73 & 100.0\% \\ 
 & & 128 & 17.76 & 187.88 & 50.85 & 2.77 & 261.55 & 138.16\% \\ 
 & & 256 & 12.1 & 88.7 & 17.54 & 1.96 & 121.32 & 148.93\% \\ 
 & & 512 & 8.83 & 45.75 & 6.19 & 0.97 & 62.47 & 144.61\% \\ 
 & & 1024 & 7.28 & 29.12 & 2.8 & 0.71 & 40.25 & 112.22\% \\ 
 \midrule
Sphere (256) & 1e-6 & 64 & 70.4 & 773.39 & 229.38 & 8.18 & 1084.9 & 100.0\% \\ 
 & & 128 & 43.21 & 334.31 & 73.89 & 4.39 & 458.11 & 118.41\% \\ 
 & & 256 & 30.57 & 177.5 & 26.6 & 2.16 & 238.92 & 113.52\% \\ 
 & & 512 & 23.0 & 102.91 & 10.1 & 1.31 & 138.3 & 98.05\% \\ 
 & & 1024 & 18.71 & 69.02 & 4.86 & 1.04 & 94.36 & 71.86\% \\ 
 \midrule
Sphere (256) & 1e-8 & 64 & 167.73 & 1293.32 & 284.51 & 21.24 & 1773.42 & 100.0\% \\ 
 & & 128 & 101.58 & 657.19 & 101.22 & 8.16 & 872.38 & 101.64\% \\ 
 & & 256 & 70.41 & 359.07 & 39.89 & 2.7 & 474.19 & 93.5\% \\ 
 & & 512 & 52.31 & 220.45 & 17.93 & 1.52 & 293.57 & 75.51\% \\ 
 & & 1024 & 42.28 & 150.04 & 9.33 & 1.2 & 204.67 & 54.15\% \\ 
 \midrule
 F16 (512) & 1e-4 & 32 & 33.22 & 470.32 & 134.85 & 4.49 & 646.01 & 100.0\% \\ 
 & & 64 & 19.32 & 230.77 & 46.92 & 2.48 & 301.27 & 107.21\% \\ 
 & & 128 & 12.66 & 122.01 & 17.67 & 1.43 & 154.88 & 104.27\% \\ 
 & & 256 & 8.93 & 80.78 & 9.67 & 1.68 & 102.23 & 78.99\% \\ 
 & & 512 & 7.14 & 56.32 & 5.11 & 0.73 & 70.0 & 57.68\% \\ 
 \midrule
F16 (512) & 1e-6 & 32 & 84.26 & 895.07 & 229.44 & 5.25 & 1219.35 & 100.0\% \\ 
 & & 64 & 49.46 & 480.55 & 79.28 & 3.13 & 615.39 & 99.07\% \\ 
 & & 128 & 32.72 & 276.15 & 33.25 & 1.97 & 345.99 & 88.11\% \\ 
 & & 256 & 23.73 & 197.55 & 17.26 & 1.41 & 241.43 & 63.13\% \\ 
 & & 512 & 19.18 & 140.59 & 10.48 & 0.99 & 172.41 & 44.2\% \\ 
 \midrule
F16 (512) & 1e-8 & 32 & 195.51 & 1769.14 & 290.37 & 7.04 & 2271.94 & 100.0\% \\ 
 & & 64 & 118.07 & 944.96 & 112.72 & 4.25 & 1185.44 & 95.83\% \\ 
 & & 128 & 76.21 & 585.09 & 50.34 & 2.76 & 717.81 & 79.13\% \\ 
 & & 256 & 57.51 & 435.65 & 31.25 & 2.16 & 529.37 & 53.65\% \\ 
 & & 512 & 46.1 & 314.68 & 18.73 & 1.41 & 383.06 & 37.07\% \\ 
 \midrule
Submarine (512) & 1e-4 & 32 & 23.42 & 349.93 & 71.66 & 4.07 & 451.51 & 100.0\% \\ 
 & & 64 & 17.18 & 226.32 & 40.21 & 2.65 & 288.43 & 78.27\% \\ 
 & & 128 & 10.2 & 104.03 & 12.36 & 1.43 & 128.99 & 87.51\% \\ 
 & & 256 & 9.01 & 82.88 & 8.36 & 1.07 & 102.14 & 55.26\% \\ 
 \midrule
Submarine (512) & 1e-6 & 32 & 57.6 & 667.59 & 128.79 & 5.35 & 863.21 & 100.0\% \\ 
 & & 64 & 42.15 & 475.55 & 65.86 & 3.35 & 589.66 & 73.2\% \\ 
 & & 128 & 26.37 & 245.83 & 24.66 & 1.91 & 300.38 & 71.84\% \\ 
 & & 256 & 23.56 & 201.56 & 15.27 & 1.51 & 243.27 & 44.35\% \\ 
 \midrule
Submarine (512) & 1e-8 & 32 & 136.74 & 1281.87 & 172.27 & 8.09 & 1606.38 & 100.0\% \\ 
 & & 64 & 101.14 & 948.74 & 94.44 & 4.7 & 1154.14 & 69.59\% \\ 
 & & 128 & 61.57 & 523.28 & 37.73 & 2.66 & 628.12 & 63.94\% \\ 
 & & 256 & 55.03 & 447.74 & 25.8 & 2.01 & 532.93 & 37.68\% \\ 
\bottomrule
\end{tabular}
\label{tab:breakdowns}
\end{table}

Table~\ref{tab:breakdowns} provides a breakdown of the running times of the different components of the algorithm.
We see the following trends in the data:
\listspace\begin{itemize}
\listsep

\item
The HF-M2L and HF-L2L (the downwards pass in the high-frequency regime) constitute the bulk of the running time.
M2L is typically the most expensive translation operator~\cite{Ying:2003}, and the interaction lists are larger in the high-frequency regime than in the low-frequency regime.

\item
Communication time is small.
The high-frequency and low-frequency communication combined are about 1\% of the running time.
For the number of processes in these experiments, this problem is expected to be compute bound.
This agrees with studies on communication times in KI-FMM~\cite{Chandramowlishwaran:2012}.

\item
The more accurate solutions do not scale as well as less accurate solutions.
Higher accuracy solutions are more memory constrained, and we speculate that lower accuracy solutions benefit relatively more from cache effects when the number of cores increases.

\end{itemize}\listspace

The data show super-linear speedups for the sphere and F16 geometries at some accuracies.
This is partly caused by the sub-optimality of the partitioning algorithm (Algorithm~\ref{alg:comp_partition}).
However, the speedups are still encouraging.

\section{Conclusion and future work}

We have presented the first parallel algorithm for the directional
multilevel approach for $N$-body problems with the Helmholtz kernel.
Under proper load balancing, the computational complexity of the
algorithm is $\bigoh{N \log N / p}$, and the communication time is
$\bigoh{p \alpha + \frac{1}{p}N \log N \beta}$.  Communication time is
minimal in our experiments.  The number of processes is limited to
$\bigoh{\sqrt{N}}$, so the running time can essentially be reduced to
$\bigoh{\sqrt{N}\log N}$.  The algorithm exhibits good strong scaling
for a variety of geometries and target accuracies.

In future work, we plan to address the following issues:
\begin{itemize}
\item Partitioning the computation at a fixed level in the octree avoids communication at high levels in the octree but limits the parallelism.
In future work, we plan to divide the work at lower levels of the octree and introduce more communication.
Since communication is such a small part of the running time in our algorithm, it is reasonable to expect further improvements with these changes.

\item Our algorithm uses $\bigoh{w^2}$ wedges for boxes of width $w$.
This makes pre-computation expensive.
We would like to reduce the total number of wedges at each box width.


\item Our implementation uses a simple octree construction.  We would
  like to incorporate the more sophisticated approaches of software
  such as \texttt{p4est}~\cite{Burstedde:2011} and
  \texttt{Dendro}~\cite{Sampath:2008} in order to improve our parallel
  octree manipulations.  This would alleviate memory constraints, since
  octree information is currently redundantly stored on each process.
\end{itemize}

\section*{Acknowledgments}
This work was partially supported by National Science Foundation under
award DMS-0846501 and by the Mathematical Multifaceted Integrated
Capability Centers (MMICCs) effort within the Applied Mathematics
activity of the U.S. Department of Energy's Advanced Scientific
Computing Research program, under Award Number(s) DE-SC0009409.
Austin R. Benson is supported by an Office of Technology Licensing Stanford Graduate Fellowship.
  
\bibliographystyle{abbrv}
\bibliography{main-bibliography}

\end{document}